\documentclass{amsart}
\usepackage[T1]{fontenc}
\usepackage{lmodern}
\usepackage{amssymb,amsmath}
\usepackage{ifxetex,ifluatex}
\newcommand{\R}{\mathbb{R}}
\newcommand{\C}{\Bbb{C}}
\newcommand{\N}{\mathbb{N}}
\newcommand{\be}{\begin{equation}}
\newcommand{\ee}{\end{equation}}
\usepackage{amsfonts}
\usepackage{amsmath}
\usepackage{amssymb}
\usepackage{graphicx}

\newtheorem{theorem}{Theorem}[section]
\theoremstyle{plain}

\newtheorem{definition}{Definition}[section]

\newtheorem{proposition}{Proposition}[section]

\DeclareMathOperator{\ca}{\cos_\alpha}
\DeclareMathOperator{\sa}{\sin_\alpha}
\newcommand{\Ea}{E_\alpha}

\begin{document}

\title{On the Invalidity of Fourier Series Expansions of Fractional Order}
\maketitle
\begin{center}
{\bf Peter Massopust\footnote{Research partially supported by DFG grant MA5801/2-1} and Ahmed I. Zayed
\vskip 8pt
Center of Mathematics, Lehrstuhl M6
\\
Technische Universit\"at M\"unchen\\
Garching b. Munich, Germany\\
massopust@ma.tum.de\\
and\\
DePaul University \\
Chicago, IL 60614, USA \\
azayed@condor.depaul.edu }
\end{center}

\vskip0.5cm \noindent
{\bf Abstract:}
 The purpose of this short paper is to show the invalidity of a Fourier series expansion of fractional order as derived by G. Jumarie in a series of papers. In his work the exponential functions $e^{in\omega x}$ are replaced by the Mittag-Leffler functions $E_\alpha \left (i (n\omega x)^\alpha\right) ,$ over the interval $[0, M_\alpha/ \omega]$ where $0< \omega<\infty $ and $M_\alpha$ is the period of the function $E_\alpha \left( ix^\alpha\right),$ i.e., $E_\alpha \left( ix^\alpha\right)=E_\alpha \left( i(x+M_\alpha)^\alpha\right).$
\vskip0.25cm
\noindent
{\bf Keywords and Phrases:} \noindent {\small Fractional derivative, Fourier series of fractional order, Mittag-Leffler function.}

\vskip0.5cm

\section{Introduction}
The Mittag-Leffler function which is named after the Swedish mathematician Mittag-Leffler, who introduced it in 1903 \cite{Mittag}, has recently been the object of many research papers because it plays an important role in Fractional Calculus. It appears in solutions of some fractional differential and integral equations and as a result it has
found some physical applications \cite{Main}.

 The Mittag-Leffler function  is an entire function of order $1/\alpha, \alpha>0$, defined by the power series
 $$  E_\alpha (z) := \sum_{k=0}^\infty \frac{z^k}{\Gamma (1 + k \alpha)},\quad  z\in \C .$$
 It is easy to see that

 $$E_1(z)=e^z, \quad E_2(z^2)= \cosh z.  $$ Because the Mittag-Leffler function generalizes the exponential function, one wonders if it generalizes or at least shares some properties of the exponential function. Several people explored that idea, among them is G. Jumarie, who derived the following proposition \cite{Jum08}	
 \begin{proposition}
Suppose that $f:\R^+_0\to \R$ is continuous and possesses for all $k\in \N$ Riemann-Liouville fractional derivatives of order $k \alpha$, with $0 < \alpha \leq 1$. Let $h > 0$. Then
\be\label{taylor}
f(x + h) = \sum_{k=0}^\infty \frac{h^{k \alpha}}{\Gamma (1 + k \alpha)}\,f^{(k \alpha)} (x),
\ee
where $D^\alpha$ is the modified Riemann-Liouville fractional derivative.
\end{proposition}
He used this proposition to show that
\be
f(x + h) = E_\alpha (h^\alpha D^\alpha) f (x),
\ee
which he
 then employed  to establish that
\be
D^\alpha E_\alpha (\lambda x^\alpha) = \lambda E_\alpha (\lambda x^\alpha), \quad x \geq 0.
\ee
This equation is a generalization of the relation
$$ \frac{d}{dx} e^{\lambda x}=\lambda e^{\lambda x},$$
and reduces to it when $\alpha =1.$
Moreover, he also showed that the semi-group property $e^{\lambda x} e^{\lambda y} =e^{\lambda (x+y)}$ may be extended to
\be
\Ea (\lambda x^\alpha) \Ea (\lambda y^\alpha) = \Ea(\lambda (x+y)^\alpha), \quad 0 < \alpha \leq 1,\;\; \lambda\in \C. \label{semigp}
\ee
In fact, in \cite{Jum12} he gave several proofs of this relation.

Analogous to the definition of $\cos x$ and $\sin x$ in terms of the exponential function, Jumarie defined
 $\ca (x^\alpha) $ and $\sa (x^\alpha) $ by
\[
E_\alpha  (i x^\alpha) = \ca (x^\alpha) + i \sa (x^\alpha), \quad x\geq 0.
\]
which leads to
$$ \ca^2 (x^\alpha) + \sa^2 (x^\alpha) =  \Ea(i x^\alpha) \Ea(-i x^\alpha). $$
Therefore,
$$ \ca^2 (x^\alpha) + \sa^2 (x^\alpha) = 1 ,$$ is equivalent to
$$  \Ea(i x^\alpha) \Ea(-i x^\alpha)=1, $$
that is
 \begin{equation}
 (\Ea (i x^\alpha))^{-1} = \Ea(-i x^\alpha). \label{Eq1}
 \end{equation}
But here we note that this relation does not follow from the functional equation
$$  \Ea (\lambda x^\alpha) \Ea (\lambda y^\alpha) = \Ea(\lambda (x+y)^\alpha) $$
because if we put $y=-x,$ we have
$$  \Ea (\lambda x^\alpha) \Ea (\lambda (-x)^\alpha) = \Ea(\lambda (x-x)^\alpha)=\Ea(0)=1 $$
which implies that
\be
(\Ea (\lambda x^\alpha))^{-1} = \Ea (\lambda (-x)^\alpha) ,
\ee
and when combined with (\ref{Eq1}) leads to
\be
(\Ea (i x^\alpha))^{-1} = \Ea (i (-x)^\alpha) = E_\alpha (-ix^\alpha).\label{Eq2}
\ee

Under the assumption that there exists a real number $M_\alpha>0$ such that
 $ \Ea (i M_\alpha^\alpha)$ $=1$,
 the periodicity of the function $E_\alpha (ix^\alpha),$ and hence, the periodicity of $\ca (x^\alpha)$ and $\sa (x^\alpha),$ will follow
  since
$$  \Ea (i x^\alpha)= \Ea (i x^\alpha) \Ea (iM_\alpha^\alpha) = \Ea (i (x+M_\alpha))^\alpha ) .$$

 The orthogonality of $\ca (mx)^\alpha $ and $\sa (mx)^\alpha$ now follows from the periodicity of these functions together with the fact that
$$ \ca^2 (x^\alpha) + \sa^2 (x^\alpha) =1.$$
It was also shown that $\ca (mx)^\alpha $ and $\sa (mx)^\alpha$  form an orthogonal basis for $L^2[0, M_\alpha].$ In \cite{Jum08}, Sec. 6, he obtained an expansion in terms of fractional sine and cosine functions as follows. Let $f$ be a periodic function with period $M_\alpha/\omega.$ Then $f$ can be expanded in a Fourier-type series of the form
$$f(x)=a_0/2 +\sum_{n=1}^\infty a_n \cos_\alpha (n \omega x)^\alpha + b_n \sin_\alpha (n \omega x)^\alpha , $$ 	
where
$$ a_n=2 \left( \frac{\omega}{M_\alpha}\right)^\alpha \int_0^{M_\alpha/\omega} f(x)\cos_\alpha (n \omega x)^\alpha (dx)^\alpha ,$$
and
$$ b_n=2\left( \frac{\omega}{M_\alpha}\right)^\alpha \int_0^{M_\alpha/\omega} f(x)\sin_\alpha (n \omega x)^\alpha (dx)^\alpha .$$
Parseval's relation takes the form
$$ \int_0^{M_\alpha/\omega} f^2(x)(dx)^\alpha=\frac{1}{2}\left( \frac{M_\alpha}{\omega}\right)^\alpha \sum_{n=0}^\infty \left( a_n^2 +b_n^2 \right). $$

These results paved the way to developing a parallel approach to harmonic analysis based on these fractional cosine and sine functions.

Fascinated with these results, the authors of this article tried to derive new results using these fractional trigonometric functions and their orthogonality. However, they soon realized that
there were some errors in Jumarie's derivations and some of the above relations cannot hold. The purpose of this paper is to show exactly that.

\section{Preliminary Result}
In this section we show by  examples that some of Jumarie's results are invalid. First, we observe from Eq. \eqref{Eq2} that we should have $\Ea (i (-x)^\alpha)=\Ea (-i x^\alpha).$
But from the definition
$$
E_\alpha (\lambda x^\alpha) = \sum_{k=0}^\infty \frac{\lambda^k x^{\alpha k}}{\Gamma (1 + k \alpha)},
$$
we have
$$ E_\alpha (-i x^\alpha) = \sum_{k=0}^\infty \frac{(-1)^k i^k x^{\alpha k}}{\Gamma (1 + k \alpha)}$$
and
$$ E_\alpha (i (-x)^\alpha) = \sum_{k=0}^\infty \frac{(-1)^{\alpha k} i^k x^{\alpha k}}{\Gamma (1 + k \alpha)},$$
and for the last two representations to be equal, we must have $(-1)^k=(-1)^{\alpha k}$ which happens only if $\alpha=1.$

We can also show directly that $\Ea(i x^\alpha) \Ea(-i x^\alpha)\neq 1$. For,
\begin{eqnarray*}
\Ea(i x^\alpha) \Ea(-i x^\alpha)&=&\left( \sum_{k=0}^\infty \frac{(ix^\alpha)^k}{\Gamma (1+\alpha k)}\right)\left( \sum_{m=0}^\infty \frac{(-ix^\alpha)^m}{\Gamma (1+\alpha m)}\right)\\
&=& \left( \sum_{k,m=0}^\infty \frac{(-1)^m(ix^\alpha)^{k+m}}{\Gamma (1+\alpha m)\Gamma (1+\alpha k)}\right)\\
&=& \sum_{n=0}^\infty i^n x^{\alpha n}\sum_{k=0}^n \frac{ (-1)^{n-k}}{\Gamma (1+ \alpha k) \Gamma (1+ \alpha (n-k))}\\
&=& \sum_{n=0}^\infty i^n x^{\alpha n}A_n(\alpha),
\end{eqnarray*}
where
$$ A_n(\alpha)=\sum_{k=0}^n \frac{ (-1)^{n-k}}{\Gamma (1+ \alpha k) \Gamma (1+ \alpha (n-k))}.$$
It is easy to see that $A_0 =1, A_{2k+1}=0$, $\forall k \in \N_0$. However,
\begin{eqnarray*}
A_2 &=& \frac{1}{\Gamma (1+2\alpha )} -\frac{1}{\Gamma^2 (1+\alpha )}+ \frac{1}{\Gamma (1+2\alpha )} \\
&=& \frac{2}{\Gamma (1+2\alpha )}-\frac{1}{\Gamma^2 (1+\alpha )}\\
&=& \frac{2\Gamma^2 (1+\alpha )-\Gamma (1+2\alpha )}{\Gamma (1+2\alpha )\Gamma^2 (1+\alpha )}\neq 0.
\end{eqnarray*}
In fact, all the even $A_{2k}\neq 0,$ but all of them are equal to zero only if $\alpha =1.$

In Figure \ref{fig1}, the product $p(x,\alpha) := \Ea(i x^\alpha) \Ea(-i x^\alpha)$ is plotted for several values of $0 < \alpha \leq 1$.
\begin{figure}[h!]
\includegraphics[width=9cm,height=4.5cm]{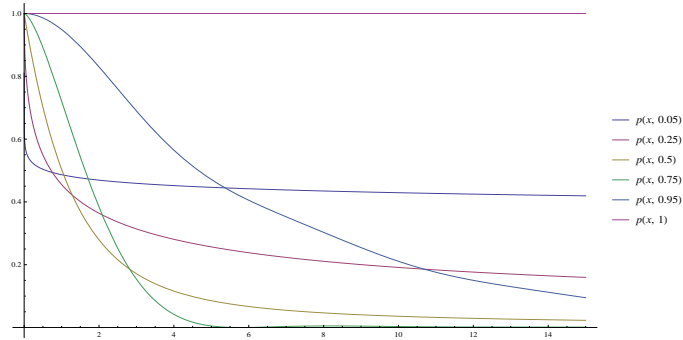}
\caption{The product $\Ea(i x^\alpha) \Ea(-i x^\alpha)$ for several values of $0 < \alpha \leq 1$.}\label{fig1}
\end{figure}

Moreover, if there existed a real number $M>0$ such that $\Ea (i M^\alpha)=1,$ we must have $ \ca (M^\alpha)=1, \; \sa (M^\alpha)=0.$ But for example, for  $\alpha = \frac12,$ we have
$$\cos_{1/2} (x^{1/2})=e^{-x} \mbox {  and  } \sin_{1/2} (x^{1/2})=\sum_{k=0}^\infty (-1)^k \frac{x^{k+1/2} }{\Gamma (k+3/2)}.$$
Using the Pochammer notation, we have $\Gamma (k+3/2)=\Gamma(3/2)(3/2)_k$ and hence
\begin{eqnarray*}
 \sin_{1/2} (x^{1/2})& = & \frac{x^{1/2}}{\Gamma (3/2)} \sum_{k=0}^\infty (-1)^k \frac{x^{k} }{(3/2)_k} \\
 &=& \frac{x^{1/2}}{\Gamma (3/2)} \sum_{k=0}^\infty  \frac{(1)_k (-x)^{k} }{(3/2)_k (1)_k}\\
 &=& \frac{x^{1/2}}{\Gamma (3/2)} \sum_{k=0}^\infty \frac{(1)_k (-x)^{k} }{(3/2)_k k!}=\frac{x^{1/2}}{\Gamma (3/2)} \Phi(1,3/2; -x),
 \end{eqnarray*}
 where $\Phi (a,b; x)$ is the confluent hypergeometric function. Using the relation
 $$ \Phi (a,b; x)=e^x\Phi (b-a,b; -x) $$ we obtain
 $$ \Phi (1,3/2; -x)=e^{-x}\Phi (1/2,3/2; x) ,$$ or
 $$\sin_{1/2} (x^{1/2})= \frac{x^{1/2}}{\Gamma (3/2)} e^{-x}\Phi (1/2,3/2; x). $$
 But in view of the relationship between the confluent hypergeometric and the error function
 $$ \mbox{Erf}(z)=z\Phi \left( 1/2, 3/2; -z^2 \right)=\int_0^z e^{-t^2}dt , $$ it is clear that
the only solution for $\cos_{1/2} (M^{1/2})=1$ and $\sin_{1/2} (M^{1/2})=0$ is $M=0,$ i.e., there is no period for these fractional trigonometric functions for $\alpha = \frac12$.

In the next section we show that our results hold not only for $\alpha =1/2$ but for all $0 < \alpha < 1. $

\section{The Main Result}

In this section we extend the  results of the previous section from $\alpha =1/2$  to all $0 < \alpha < 1. $
To this end, we need the following definition.

\begin{definition}
A function $f: \R^+\to \R$ is called completely monotonic (CM)
if $f\in C^\infty$ and $\forall\,n\in \N_0\;\forall\,x\in \R^+:\;(-1)^n f^{(n)}(x) \geq 0$.
\end{definition}

\begin{theorem}
Let $0 \leq \alpha < 1$. Then there does not exist an $M_\alpha > 0$ so that
\[
\Ea (i M_\alpha^\alpha) = 1\quad\text{or, equivalently,}\quad \ca (M_\alpha^\alpha) = 1\;\;\text{and}\;\;\sa(M_\alpha^\alpha) = 0.
\]
Moreover, the relation
\be
\Ea (\lambda x^\alpha) \Ea (\lambda y^\alpha) = \Ea(\lambda (x+y)^\alpha), \quad 0 < \alpha \leq 1,\;\; \lambda\in \C,\label{Eq3}
\ee
can hold only if $\alpha = 1$.
\end{theorem}
\begin{proof}
We first note that $\ca (x^\alpha) = E_{2\alpha} (-x^{2\alpha})$. This is verified by direct computation or by referring to the duplication formula \cite[Eqn. 2.14]{M}. It suffices to show that $\ca (x^\alpha) < 1$ for all $x > 0$.

To this end, we use a result from \cite{P} which states that the Mittag-Leffler function $E_{\alpha} (-x)$ is completely monotonic for all $0\leq \alpha\leq 1$ and $x\geq 0$. As $x\mapsto x^\alpha$ is a Bernstein function for $0 < \alpha < 1$ and the composition $E_{\alpha} (- \,\cdot) \circ (\cdot)^\alpha$ of the CM function $E_{\alpha} (-\,\cdot)$ with the Bernstein function $x\mapsto x^\alpha$ is again a CM function \cite[Theorem 3.7. (ii)]{Schill}, we see that $E_{2\alpha} (-x^{2\alpha})$ is a CM function for $0 < \alpha < \frac12$.

As $E_{2\alpha} (0) = 1$ and $E_{2\alpha} (-x^{2\alpha})\to 0+$ (see, for instance, \cite[Eq. (3.3)]{Main2}), the complete monotonicity for $ 0 < \alpha < \frac12$ and $x > 0$ now implies that there cannot exist an $M_\alpha > 0$ such that $\ca (M_\alpha^\alpha) = 1$.

Now suppose $\frac12 < \alpha < 1$. It is shown in \cite[Section 5.2]{M} that in this case the Mittag-Leffler function $E_{2\alpha}$ can be written as a sum of two functions, $f_{2\alpha}$ and $g_{2\alpha}$, where the former is CM and the latter oscillatory. More precisely,
\[
E_{2\alpha} (-x^{2\alpha}) = f_{2\alpha}(-x^{2\alpha}) + g_{2\alpha}(-x^{2\alpha}), \quad \frac12 < \alpha < 1, \;\;x \geq 0,
\]
where
\begin{align*}
f_{2\alpha}(-x^{2\alpha}) &:= \frac{1}{\pi}\,\int_0^\infty \frac{e^{-s x^{2\alpha}}\,s^{2\alpha}\sin(2\alpha\pi)}{s^{4\alpha} + 2 s^{2\alpha} \cos(2\alpha\pi) + 1}\,ds,
\end{align*}
and
\begin{align*}
g_{2\alpha}(-x^{2\alpha}) &:= \frac{2}{\alpha}\,e^{-x^{2\alpha}\,\cos (\pi/2\alpha)}\,\cos \left(x^{2\alpha}\,\sin \frac{\pi}{2\alpha}\right).
\end{align*}
The function $f_{2\alpha}$ is CM, satisfies $f_{2\alpha} (0) = 1 - \frac2\alpha$, and increases towards zero from above, whereas the function $g_{2\alpha}$ is oscillatory with exponentially decaying amplitude and $g_{2\alpha} (0) = \frac{2}{\alpha}$. Hence, $E_{2\alpha} (-x^{2\alpha}) < 1$, for all $\frac12 < \alpha < 1$ and $x > 0$. Thus, there cannot exist an $M_\alpha$ in this case either.

The case $\alpha = \frac12$ was considered above.

Next we show that the purported functional equation \eqref{Eq3} can hold only if $\alpha =1.$
First note that $\Ea (\lambda x^\alpha) = \Ea \circ h (x)$, where $h(x) := \lambda x^\alpha$. Writing $F_\alpha$ for the composition $\Ea \circ h$, Eq. \eqref{Eq3} reads
\be
F_\alpha(x) F_\alpha (y) = F_\alpha (x+y).\label{Eq4}
\ee
However, the only nonzero continuous solutions of Eq. \eqref{Eq4} are exponential functions of the form $e^{c x}$, where $c\in \C$; see for instance \cite[Chapter 2]{Acz}. Therefore, $F_\alpha (x) = \Ea (\lambda x^\alpha) = e^{c x}$, for all $x$. Successively differentiating the power series for $\Ea (\lambda x^\alpha) $ and $e^{c x}$ and letting $x\to 0$, shows that $\lambda = c$ and $\alpha = 1$.
\end{proof}

Figure \ref{fig2} below depicts some graphs of $\ca (x^{\alpha}) = E_{2\alpha} (-x^{2\alpha})$, for $0 < x \leq 1$. The case $\alpha := 1$ produces the cosine function.
\begin{figure}[h!]
\includegraphics[width=10cm,height=4.5cm]{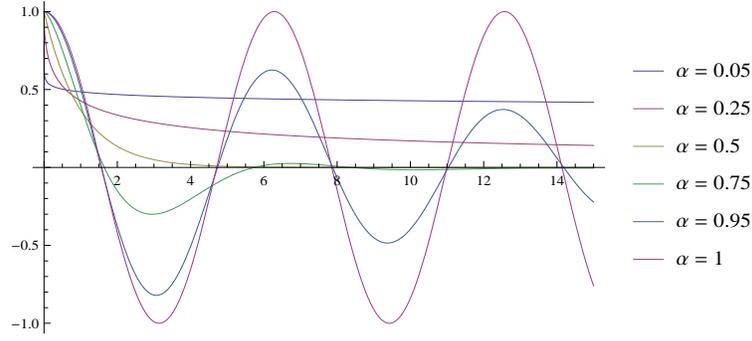}
\caption{$\ca (x^\alpha)$ for several values of $0 < \alpha \leq 1$.}\label{fig2}
\end{figure}

In \cite{Jum12}, the purported functional equation for $\Ea (\lambda x^\alpha)$ is derived using several approaches. One of them is the product rule
\be\label{eq10}
D^\alpha (f g) = g D^\alpha f + f D^\alpha g.
\ee
However, this rule is not correct.

To see this, let $g := f := x^{1/2}$. It is straight-forward to establish that for $x \geq 0$,  $p > -1$, and $0 < \alpha < 1$, the fractional derivative of $x^{p}$ is given by
\[
D^\alpha x^p = \frac{\Gamma (1 + p)}{\Gamma (1 + p - \alpha)}\, x^{p - \alpha}.
\]
Hence, the left-hand side of \eqref{eq10} computes to
\[
D^\alpha (x) =  \frac{x^{1 - \alpha}}{\Gamma (2 - \alpha)},
\]
whereas the right-hand side equals
\[
2 x^{1/2} D^\alpha (x^{1/2}) = \frac{\sqrt{\pi}\,x^{1-\alpha}}{\Gamma(\frac32 - \alpha)}.
\]
Both sides are identical only if $\alpha = 1$.

Similarly, the two purported chain rules, which are also employed in the derivation of \eqref{Eq3}, namely
\[
D^\alpha (f\circ u) (x) = \frac{d f}{du} \,D^\alpha u (x)\quad\text{ and }\quad D^\alpha (f\circ u) (x) = (D^\alpha_u f) \cdot \left(\frac{du}{dx}\right)^\alpha,
\]
are not correct. To validate this, take for the former, $f(x) := x^2$ and $u(x) := x^{1/2}$, $x\geq 0$. Then, for any $0 < \alpha\leq 1$, one has
\be\label{eq11}
D^\alpha (f\circ u) (x) = D^\alpha x = \frac{x^{1 - \alpha}}{\Gamma (2 - \alpha)}
\ee
and
\be\label{eq12}
\frac{d f}{du}\bigg\lvert_{u = x^{1/2}} \cdot D^\alpha u (x) = 2 x^{1/2} D^\alpha x^{1/2} = \frac{\sqrt{\pi}\,x^{1 - \alpha}}{\Gamma (\frac32 - \alpha)},
\ee
and this two expressions are identical only if $\alpha = 1$. Now for the latter, take $f(u) := u^{1/2}$ and $u(x) := x^2$, $x \geq 0$. Then, for any $0 < \alpha \leq 1$, we again have
\[
D^\alpha (f\circ u) (x) = D^\alpha x = \frac{x^{1 - \alpha}}{\Gamma (2 - \alpha)},
\]
whereas now
\be\label{eq13}
(D^\alpha_u f)\big\lvert_{u = x^2} \cdot \left(\frac{du}{dx}\right)^\alpha = D^\alpha_u (u^{1/2})\big\lvert_{u = x^2} (2x)^\alpha = \frac{2^{\alpha -1}\,\sqrt{\pi}\,x^{1-\alpha}}{\Gamma(\frac32 - \alpha)}.
\ee
Eqns. \eqref{eq11}, \eqref{eq12}, and \eqref{eq13} are identical only if $\alpha = 1$.

For illustrative purposes, we plotted the difference $\Ea(\lambda (x+y)^\alpha) - \Ea (\lambda x^\alpha) \Ea (\lambda y^\alpha)$  in Figure \ref{fig3} for $(x,y) \in [0,2]^2$, $\lambda = 1$, $\alpha = \frac14$ (left), and $\alpha=\frac34$ (right).

\begin{figure}[h!]
\includegraphics[width=5cm,height=3cm]{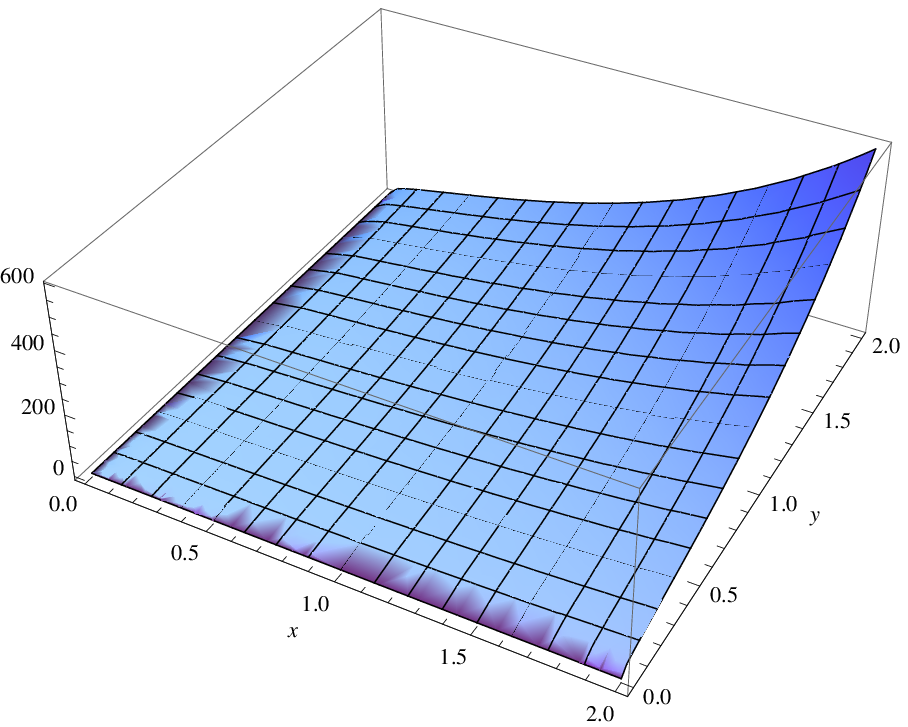}\hspace{1cm}
\includegraphics[width=5cm,height=3cm]{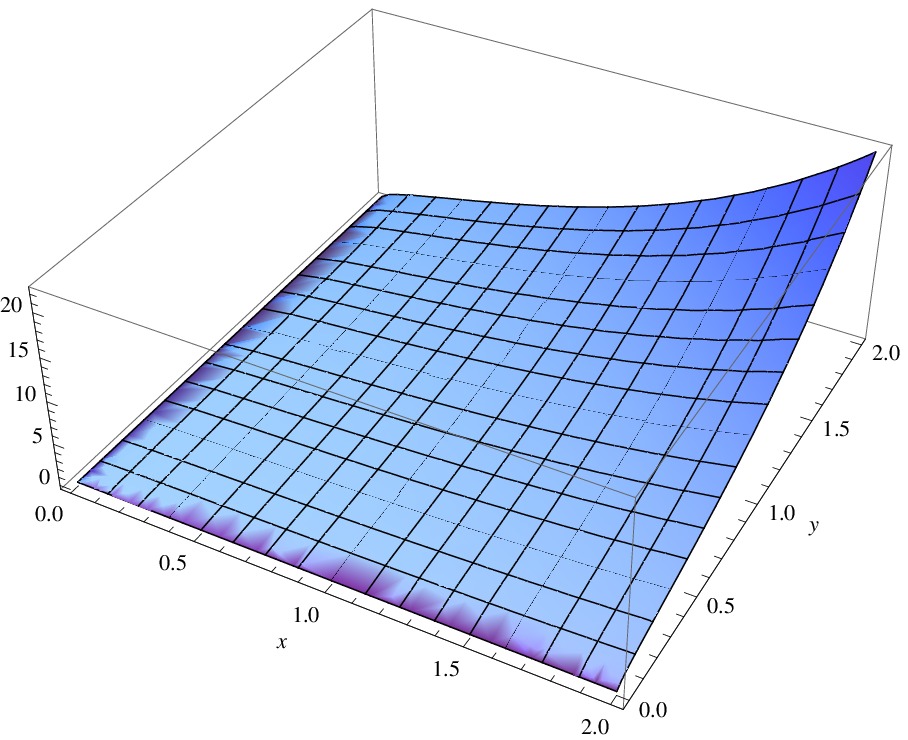}
\caption{The difference $\Ea(\lambda (x+y)^\alpha) - \Ea (\lambda x^\alpha) \Ea (\lambda y^\alpha)$.}\label{fig3}
\end{figure}

\end{document}